\documentclass[11pt]{article}
\title{The valuative capacity of the set of sums of $d$-th powers\footnote{\copyright 2016. This manuscript version is made available under the CC-BY-NC-ND 4.0 license \url{http://creativecommons.org/licenses/by-nc-nd/4.0/}}}
\usepackage[british]{babel}
\usepackage{datetime}
\usepackage[utf8]{inputenc}
\date{\today}
\author{B.Langlois, Marie-Andrée}

\usepackage{hyperref}
\usepackage{geometry}
\usepackage{amsthm}
\usepackage{amsmath}
\usepackage{amssymb}
\usepackage{tikz}
\usepackage{stmaryrd}
\usepackage{verbatim}
\usepackage{color}
\usepackage{mathtools}
\usepackage{fancyhdr}

\usepackage[alphabetic]{amsrefs}

\newcommand{\ZZ}{\mathbb{Z}}

\newcommand{\QQ}{\mathbb{Q}}

\newcommand{\NN}{\mathbb{N}}

\newcommand{\bard}{\overline{D}}

\newcommand{\invlim}[1]{\lim\limits_{\overleftarrow{#1}}}

\usepackage[all]{xy}

\newtheorem{Th}{Theorem}
\newtheorem{Def}[Th]{Definition}
\newtheorem{Lemma}[Th]{Lemma}
\newtheorem{Ex}[Th]{Example}
\newcommand{\EX}{\begin{Ex}\normalfont}
\newtheorem{Cor}[Th]{Corollary}
\newtheorem{Prop}[Th]{Proposition}

\begin{document}
\maketitle

\begin{abstract}
If $E$ is a subset of the integers then the $n$-th characteristic ideal of $E$ is the fractional ideal of $\ZZ $ consisting of $0$ and the leading coefficients of the polynomials in $\QQ[x]$ of degree no more than $n$ which are integer valued on $E$. For $p$ a prime the characteristic sequence of $Int(E,\ZZ)$ is the sequence $\alpha_E (n)$ of negatives of the $p$-adic valuations of these ideals. The asymptotic limit $\lim_{n\to \infty}\frac{\alpha_{E,p}(n)}{n}$ of this sequence, called the valuative capacity of $E$, gives information about the geometry of $E$. We compute these valuative capacities for the sets $E$ of sums of $\ell \geq 2$ integers to the power of $d$, by observing the $p$-adic closure of these sets. 
\end{abstract}

\section{Introduction}

\noindent
Given $E$ a subset of $\ZZ$, the valuative capacity of $E$ is a notion that was first introduced by Chabert in \cite{Chabert}, in analogy to the idea of capacity of a subset originally introduced by Fekete in 1923 in \cite{Fekete}. Recent results \cite{Fares2} show that these notions actually coincide in many cases. The later has played a central role in several important results such as the Polya-Szeg\"{o} theorem \cite{PS}
, integer polynomials approximation \cite{Ferguson} and algebraic geometry \cite{Rumely}. Chabert's definition is by way of the theory of integer valued polynomials:\ \\

\noindent
\begin{Def}
For any subset $E$ of $\ZZ$ the ring of integer valued polynomials on $E$ is defined to be $$Int(E,\ZZ)=\{ f(x)\in \QQ[x]\ | \ f(E)\subseteq \ZZ \}.$$
\end{Def}

\noindent
\begin{Def}
The sequence of characteristic ideals of $E$ is $\{ I_n \ | \ n=0,1,2,\ldots \}$ where $I_n$ is the fractional ideal formed by $0$ and the leading coefficients of the elements of $Int(E,\ZZ)$ of degree no more than $n$ and the characteristic sequence of $E$ with respect to the prime $p$, is the sequence of negatives of the $p$-adic valuations of these ideals, denoted $\alpha_{E,p}(n)$. 
\end{Def}

\noindent
The valuative capacity arises from wanting to find the asymptotic behaviour of $\alpha_{E,p}(n)$. In \cite{Chabert} Chabert shows that the limit of $\frac{\alpha_{E,p}(n)}{n}$ with respect to $n$ exists, and defines:\ \\

\noindent
\begin{Def}\label{Def:valcap}
The valuative capacity of $E$ with respect to $p$ is the following limit: $$L_{E,p}=\lim_{n\to \infty}\frac{\alpha_{E,p}(n)}{n} .$$
\end{Def}\ \\

\noindent
In 1997, Bhargava introduced the following definition which is very important when studying integer valued polynomials: 

\noindent
\begin{Def}
A $p$-ordering of $E$ is a sequence $(a_n)_{n\geq 0}$, such that, for each $n$, $a_n\in E$ is chosen to minimize
$$ \nu_p((a_{n}-a_{n-1})\cdots (a_n-a_0)),$$ where $\nu_p$ denotes the $p$-adic valuation.
\end{Def}

\begin{Prop}\cite{Bh}
Let $(a_n)_{n\geq 0}$ be a sequence of distinct elements of $E$. Then, $(a_n)_{n\geq 0}$ is a $p$-ordering of $E$ if and only if for a given $0\leq n$, the polynomials $$f_n(X) = \prod_{k=0}^{n-1} \frac{X-a_k}{a_n-a_k} $$ form a basis for the $\ZZ_{(p)}$-module $Int (E,\ZZ_{(p)})=\{ f(x)\in \QQ[x]\ | \ f(E)\subseteq \ZZ_{(p)} \}$. Consequently, $\nu_p \left( \prod_{k=0}^{n-1} (a_n-a_k) \right)=\alpha_{E,p}(n)$ for $0\leq n \leq m$.
\end{Prop}

\noindent
In this paper we are interested in finding the valuative capacity of the set of sums of $\ell\geq 2$ integers which are each $d$-th powers, for $d\geq 3$, 
since the details for the case $d=2$ are in \cite{Fares} and those for $\ell=1$ are in~\cite{Fares1}.\ \\

\noindent
\begin{Def}\label{Def:elld}
For a fixed $d\in \ZZ$ with $d\geq 0$, we define $D$ to be the set of $d$-th powers of integers, thus $D=\{x^d \ | \ x \in \ZZ \}$ and we let $\ell D= D+\cdots + D$, for $\ell$ terms in the sum. 
\end{Def}\

\noindent
The main result of this paper is:

\noindent
\begin{Th}
Let $p$ be a prime number, $d$ a positive integer and $\ell$ an integer greater or equal to 2. Then, $L_{\ell D,p}$ is an algebraic number of degree at most 2. Moreover, if 0 can be written non-trivially modulo $p^e$ as a sum of $\ell$ elements to the power of $d$, where $e=1+2\nu_p (d)$, then $L_{\ell D,p}$ is a rational number. 
\end{Th}

\noindent
We will divide the paper into the following sections: first we will go over background and notation, where we prove some general results about valuative capacity that will be needed, then we will prove the main theorem, and then discuss the cases where we know that we can write $0$ as a non-trivial sum of $\ell$ elements to the power of $d$, and give formulas for the valuative capacity in those cases. 

\section{Background and Notation}

In this work we are interested in the sets $\ell D$, for $d$ a positive integer with $d>2$. Similarly to Definition~\ref{Def:elld}:

\noindent
\begin{Def}
Let $D_{p^e}$ denote the set of $d$-th powers modulo $p^e$, for $e\geq 1$ and $\ell D_{p^e}$ the sets of sums of $\ell$ elements to the power of $D$ modulo $p^e$. We will also make use of $\bard=\invlim{m\in \NN} D_{p^m}$, the $p$-adic closure of $D$ in $\hat{\ZZ}_p$, and similarly we will consider $\overline{\ell D}$.
\end{Def}

\noindent
We will now recall some propositions that will help us to compute valuative capacities.

\noindent
\begin{Prop}\label{Prop:integers}
For a prime $p$, the valuative capacity of the set of integers is $ L_{\ZZ,p}=\frac{1}{p-1} .$
\end{Prop}
\begin{proof}
The positive integers in increasing order are a $p$-ordering of $\ZZ$, hence, by Definition~\ref{Def:valcap}, we have that $\alpha_{\ZZ,p}(n)=\nu_p(n!)$. By Legendre's formula $\nu_p(n!)=\frac{n-\sum n_i}{p-1}$, where $0\leq n_i < p$ are the coefficients of the base $p$ expansion of $n$, i.e. $n= \sum n_i p^i$. We can thus compute $$ L_{\ZZ,p}=\lim_{n\to \infty} \frac{\alpha_{\ZZ,p}(n)}{n}=\frac{1}{p-1} .$$
\end{proof}

\noindent
Given $A$ a subset of the integers, for the remainder of this paper, $\overline{A}$ will denote the $p$-adic closure of $A$ in $\hat{\ZZ}_{p}$. Also note that $\ell D$ is the set the previously defined of sums of $\ell$ elements to the power of $d$, but for a given integer $k$, a prime $p$, and $E\subseteq \ZZ$, $p^kE$ is the usual set $\{ p^k a\ | \ a\in E\}$.

\noindent
\begin{Prop}\label{Prop:cosets} Let $p$ be a fixed prime and $A$ be a subset of $\ZZ$.\\
\begin{enumerate}
\item ~\cite{CB} We have that $L_{\alpha_{\overline{A},p}}=L_{\alpha_{A,p}}$, since $\alpha_{\overline{A},p}=\alpha_{A,p}$.
\item ~\cite{Johnson1} If $A$ has characteristic sequence $\alpha_{A,p}(n)$ then for any $c\in \ZZ$ the characteristic sequence of $A+c$ is also $\alpha_{A,p}(n)$ and the characteristic sequence of $p^kA$ is $\alpha_{A,p}(n)+kn$.
\item ~\cite{Johnson1} If $B$ is another subset of  $\ZZ$, with the property that for any $x\in A$ and $y\in B$ it is the case that $\nu_p (x-y)=0$, then the characteristic sequence of $A\cup B$ is the disjoint union of the sequences $\alpha_{A,p}(n)$ and $\alpha_{B,p}(n)$ sorted into nondecreasing order. 
\end{enumerate}
\end{Prop}

\noindent
\begin{Def}
For a fixed prime $p$, and $A$, $B$ two subsets of $\ZZ$, the characteristic sequence of $A\cup B$ mentioned in Proposition~\ref{Prop:cosets}(3) is called the shuffle product of $\alpha_{A,p}(n)$ and $\alpha_{B,p}(n)$ and is denoted $(\alpha_{A,p} \wedge \alpha_{B,p})(n).$\ \\
\end{Def}

\noindent
\begin{Prop}~\cite{Johnson2}\label{Prop:shuffle} If $\alpha_{A,p}(n)$ and $\alpha_{B,p}(n)$ are the characteristic sequences of $A$ and $B$ respectively, for a prime $p$, and $A$, $B$ satisfying Proposition~\ref{Prop:cosets}(3), with $L_{A,p}=\lim_{n\to \infty}\frac{\alpha_{A,p}(n)}{n}$ and $L_{B,p}=\lim_{n\to \infty}\frac{\alpha_{B,p}(n)}{n}$ then $$\frac{1}{L_{A\cup B,p}}=\frac{1}{L_{A,p}}+\frac{1}{L_{B,p}} .$$
\end{Prop}

\noindent
The next proposition is a generalization of the above, which will prove itself to be very useful when computing valuative capacities.

\noindent
\begin{Prop}\cite{Johnson}\label{Prop:Johnson}
Given a prime $p$, if $A$ and $B$ are disjoint subsets with the property that there is a nonnegative integer $k$ such that $\nu_p (a-b)=k$ for any $a\in A$ and $b\in B$, then $$\dfrac{1}{L_{A\cup B,p}-k}=\dfrac{1}{L_{A,p}-k}+\dfrac{1}{L_{B,p}-k} .$$
\end{Prop}\ \\

\noindent
\begin{Prop}\label{Prop:union}
If $E$ is a union of cosets modulo $p^m$ for some $m$, then the valuative capacity of $E$ is rational and recursively computable.
\end{Prop}
\begin{proof}
We prove the above by induction on $m$, the case $m=1$ being Proposition~\ref{Prop:shuffle}. Suppose $L_{E,p}\in \QQ$ for all $E=\bigcup_{i=1}^\ell (a_i +p^k\ZZ)$, for $1<k<m$.\\ \\
Suppose $E=\bigcup_{i=1}^\ell (a_i +p^m\ZZ)$ and, for $j=0,1,\ldots , p-1$, let $\displaystyle E_j =\bigcup_{a_i \equiv j \pmod{p}} (a_i+p^m \ZZ)$. We have that $E=\bigcup E_j$ and $$L_{E,p}=\left( \sum_{j=0}^{p-1}\left(L_{E_j,p} \right)^{-1} \right)^{-1}$$ since the $E_j$ satisfy the hypotheses of Proposition~\ref{Prop:shuffle}. Thus $L_{E,p}$ is a rational combination of the $L_{E_j,p}$'s, which are rational by induction and Proposition~\ref{Prop:Johnson}. Each $E_j$ is the translate by $j$ of $p$ times a union of cosets $\pmod{p^{m-1}}$, so our induction hypothesis applies and $L_{E,p}\in \QQ$.
\end{proof}

\noindent
Propositions~\ref{Prop:shuffle} and \ref{Prop:Johnson} give a method of computing $L_A$ for $A=A_1\cup A_2$ in terms of $L_{A_1}$ and $L_{A_2}$ when $A_1,A_2$ are such that $\nu_p(x_1-x_2)$ is constant for $x_i \in A_i$. To handle some cases in which this conditions fails we proceed in several steps, expressing $A$ as a nested union of sets $B_i$ with  $B_k=A_k\cup B_{k+1}$ and $\nu_p(x_1-x_2)$ constant if $x_1\in A_k$ and $x_2 \in B_{k+1}$.\\

\noindent
The next propositions will involve continued fractions, and we will use the concise notation for these where $[a;a_0,a_1,\ldots, a_k]$ denotes 
$$a+\cfrac{1}{a_0+\cfrac{1}{a_1+\cfrac{}{\ddots+ \dfrac{1}{a_k}}}} $$
for $k$ a positive integer. More details about this notation can be found in \cite[IV p.81]{Davenport}. (Note that in \cite[IV p.81]{Davenport}, the $a_i$'s are integers, while in what follows they will be in $\QQ$.)\\

\noindent
Thus Proposition~\ref{Prop:Johnson} becomes: given a prime $p$, if $A$ and $B$ are disjoint subsets with the property that there is a nonnegative integer $k$ such that $\nu_p (a-b)=k$ for any $a\in A$ and $b\in B$, then $L_{A\cup B,p}$ has the continued fraction expansion: $$L_{A\cup B,p}=a_0+\cfrac{1}{a_1+\cfrac{1}{a_2+L_{B,p}}}=[a_0;a_1,a_2,a_3],$$ with $a_0=k$, $a_1=\frac{1}{L_{A,p}-k}$, $a_2=-k$ and $a_3=\frac{1}{L_{B,p}}$.\ \\

\noindent
\begin{Prop}\label{Prop:continued} 
Fix a prime $p$. Let $A_0,A_1,\ldots,A_m$ be disjoint subsets of $\ZZ$ such that, whenever $0\leq k<h\leq m$, $a\in A_k$, and $b\in A_h$, one has $\nu_p(a-b)=k$. Then, the $p$-valuative capacity of $A=A_0\cup\cdots \cup A_m$, has the following continued fraction expansion:
$$L_{A,p}=[0; a_0, a_1, \ldots, a_{2(m-1)}, a_{2m-1}] $$
where $a_{2k}=\frac{1}{L_{A_k}-k}$ for $0\leq k\leq m-1$, $a_{2k+1}=1$ for $0\leq k < m-1$, and $a_{2m-1}=L_{A_m}-(m-1)$.
\end{Prop}
\begin{proof}
We prove the statement by induction on $m$. For $m=0$, we are in the case $A=A_0$ and the continued fraction equals $L_{A_0}$. For $m=1$ we have $A=A_0\cup A_1$ and by assumption $\nu_p (a-b)=1$ if $a\in A_0$ and $b\in A_1=B$, then by Proposition~\ref{Prop:Johnson}:
\[
L_{A_0\cup A_1}=1+\cfrac{1}{\cfrac{1}{L_{A_0}-1}+\cfrac{1}{L_{A_1}-1}}.
\]
Now suppose that this results hold for $1<m\in \ZZ$, and we will prove the case $m+1$. We have that $\nu_p(a-p)=m$ if $a\in A_{m}$ and $b\in A_{m+1}$ in this case, by Proposition~\ref{Prop:Johnson}:
\[
L_{A_{m}\cup A_{M+1}}=m +\cfrac{1}{\cfrac{1}{L_{A_{m}}-m}+\cfrac{1}{L_{A_{M+1}}-m}},
\] 
by induction hypothesis
$$L_{A}=[0;a_0,a_1,a_2,\ldots, a_{2m}, L_{(A_{m}\cup A_{M+1})}-(m-1)] .$$
Substituting appropriately yields:
\begin{align*}
L_{A}&=\left[0;a_0,a_1,a_2,\ldots, a_{2m}, \left(  m +\tfrac{1}{\tfrac{1}{L_{A_{m}}-m}+\tfrac{1}{L_{A_{M+1}}-m}}  \right)   -m+1 \right]\\ \\
&=\left[0;a_0,a_1,a_2,\ldots, a_{2m},1, \tfrac{1}{L_{A_{m}}-m}+\tfrac{1}{L_{A_{m+1}-m}}   \right]\\ \\
&=\left[0;a_0,a_1,a_2,\ldots, a_{2m},a_{2m+1}, a_{2m+2}, L_{A_{m+1}}-m   \right]
\end{align*}
with $a_{2m+1}=1$ and $a_{2m+2}=\cfrac{1}{L_{A_{m}}-m}$.
\end{proof}\ \\

\noindent
If $E$ is a subset of $\ZZ$, which we can rearrange as a union of subsets $E=\left( \bigcup_{k=0}^{m-1} E_k \right) \cup p^m E$, where the $E_k$'s are unions of cosets $E_k =\bigcup (c+p^m \ZZ)$, $c\neq 0$, where $\nu_p(c)=k$, for all $c$ from $E_k$, then the previous proposition applies.

\noindent
\begin{Cor}\label{Cor:unionquadratic}
If $E=E'\cup p^m E$, where $E'$ is a union of nonzero cosets $\pmod{p^m}$, then $L_{E,p}$ is the root of a quadratic polynomial in $\QQ[x]$, whose coefficients are recursively computable.
\end{Cor}
\begin{proof}
By Proposition~\ref{Prop:union} $L_{E',p}\in \QQ$, and we can set up this situation as in Proposition~\ref{Prop:continued}, where in this case $A=E$, the sets $A_k$ are separated according to the $p$-adic valuation of their elements and the differences of these with elements of other subsets. \\ \\
Then we have that $A_m=p^mE$, hence the valuative capacity is $$L_{E,p} =[0;a_0,a_1,\ldots a_{2(m-1)},a_{2m-1} ]$$ where $a_{2k}=\frac{1}{L_{A_k,p}}$, $a_{2k+1}=1$ for $0\leq k < m-1$ and $a_{2m-1}=L_{A_m}-(m-1)=L_A-1$ by Proposition~\ref{Prop:cosets}(2). 
The $E_k$ also have rational valuative capacity by Proposition~\ref{Prop:Johnson} and $L_{E',p}\in \QQ$. Since this is a continued fraction of period $2(m-1)$, the argument in~\cite[Chapter IV section 9]{Davenport} gives that it is the root of a quadratic polynomial over $\QQ[x]$, although the $a_i$'s are not necessarily integers, the values may be rational, the result still applies.
\end{proof}\ \\

Now we look into how to rearrange a subset $E=E'\cup p^m E$ in practice. First we split $E'$ into smaller subsets $E_i$ such that $\nu (x_1-x_2)=m-1$, for all $x_1,x_2\ \in E_i$ and $x_1\neq x_2$. We then need to split up these subsets again depending on the valuation of the differences of their elements. There is no straightforward way of doing this, we illustrate the process in the following example. We then compute $L_{E_i,p}$ which is a rational number, and for the case we are interested in for our main theorem the $E_i$ are finite, since we only get a finite number of residue classes that are a sum of $\ell$ elements to the power of $d$ in $\ZZ/(p^m)$. Thus $L_{E_i,p}$ depends on $p$ and the number of elements in $E_i$ only. We then compute $\nu_p (E_i-E_i')$, for all of subsets. Now we calculate the valuative capacity for the union of subsets having the highest valuation using Proposition~\ref{Prop:Johnson}, and keep repeating the process until we can use Proposition~\ref{Prop:shuffle}.\\

\noindent
\EX \ \\
\begin{enumerate}
\item[(a)] We illustrate the above with $p=3$, and $A=\{0,1,2,3,10,11,12,19,20,21 \}+3^3\ZZ$ (this set is actually $3D_{3^3}$ when $d=6$). We write $A$ such that it satisfies the decomposition from Proposition\ref{Prop:continued}:
\begin{align*}
A_{0}&=\{1,2,10,11,19,20 \}+3^3\ZZ\\
A_{1}&=\{3,12,21 \}+3^3\ZZ\\
A_{2}&=\{0 \}+3^3\ZZ
\end{align*}
we have that $p\nmid a$ for all $a\in A_0$ and $p$ divides exactly $a$ for all $a\in A_1$. 
If $a\in A_0$, $b\in A_1\cup A_2$, then $\nu_3(a-b)=0$, since $p\nmid a$ and $p\mid b$. We can rewrite $A_0$: 
\begin{align*}
A_0&=(1+\{0,9,18 \}+3^3\ZZ)\cup (2+\{0,9,18 \}+3^3\ZZ).
\end{align*}
The valuative capacity of both sets in the union of $A_0$ is
\begin{align*}
L_{1+\{0,9,18\}+3^3\ZZ}=L_{2+\{0,9,18\}+3^3\ZZ }&=L_{\{0,9,18\}+3^3\ZZ}=L_{9(\{0,1,2\}+3\ZZ) }=2+L_{\ZZ }=2+\frac{1}{2}=\frac{5}{2}
\end{align*}
Now we can find the valuative capacity of $A_0,\ A_1$ and $A_2$ using Proposition~\ref{Prop:shuffle}, for which we obtain $L_{A_0}=\frac{5}{4}$, $L_{A_1}=\frac{5}{2}$ and $L_{A_2}=\frac{7}{2}$.
We are ready to compute the valuative capacity of $A$:
\begin{align*}
L_{A}=\cfrac{1}{\cfrac{1}{L_{A_0}}+\cfrac{1}{ 1+ \cfrac{1}{\cfrac{1}{L_{A_1}-1}+\cfrac{1}{L_{A_2}-1}} }}=\cfrac{1}{\cfrac{1}{\frac{5}{4}}+\cfrac{1}{ 1+ \cfrac{1}{\cfrac{1}{\frac{5}{2}-1}+\cfrac{1}{\frac{7}{2}-1}} }}=\frac{155}{204}.
\end{align*}
\item[(b)] Now we look into the valuative capacity of the set $E=E'\cup 3^{6} E$, where $E'$ is $A_0\cup A_1$ from part (a). We use Proposition~\ref{Prop:continued} with $A_2=3^{6}E$.\\ \\
Then we have that $L_{E}=[0;a_0,a_1,a_2,L_{A_2}-1]$, where $a_0=\frac{1}{L_{A_0}}=\frac{4}{5}$, $a_1=1$, $a_2=\frac{1}{L_{A_1}-1}=\frac{2}{3}$,  and $L_{A_2}=L_{3^{12} E}=6+L_{E}$. Hence $L_{E}=[0;\frac{4}{5},1,\frac{2}{3},L_{A_2}-1]$. Solving the continued fractions gives that $L_{E}$ is a solution to the following quadratic equation: $$30L_{E}^2+152L_{E} -140=0 .$$
This equation has for positive root $L_{E}=\frac{\sqrt{2494}}{15}-\frac{38}{15}$.
\end{enumerate}
\end{Ex}\ \\

\section{Main Theorem}

\noindent
Now we are ready to prove the main result. (Note that in saying that zero can be written non-trivially as the sum of $\ell$ elements to the power of $d$, we mean that $p$ does not divide at least one element in the sum.)

\noindent
\begin{Th}\label{Th:main}
Let $p$ be a prime number, $d$ a positive integer and $\ell$ an integer greater or equal to 2. Then, $L_{\ell D,p}$ is an algebraic number of degree at most 2. Moreover, if 0 can be written non-trivially modulo $p^e$ as a sum of $\ell$ elements to the power of $d$, where $e=1+2\nu_p (d)$, then $L_{\ell D,p}$ is a rational number. 
\end{Th}
\begin{proof}
Note that the conditions on $d$ imply that $d\geq e$. We start by looking at $$E=\left\{ [c]\in \ell D_{p^e}\ | \ [c]=\sum_{i=1}^\ell [x_i]^d, \mbox{ where at least one of the } x_i \mbox{ is not divisible by }p \right\} .$$
Without loss of generality we may assume that for $[c]\in E$, $p\nmid x_1$. Suppose that $c\in \hat{\ZZ}_p$ is such that $[c]\in E$, and that $\{x_i  \}_{i=1}^\ell\subseteq \hat{\ZZ}_p$ are such that $c\equiv \sum_{i=1}^\ell x_i^d \pmod{p^e}$. We claim that $c\in \overline{\ell D}$ in this case. Consider the polynomial $f(x)=x^d+\displaystyle \sum_{i=2}^\ell x_i^d-c$. $f$ has at least one root $\pmod{p^e}$, the integer $x_1$, with $p\nmid x_1$ and $f'(x)=dx^{d-1}$, is such that $\nu_p (f'(x_1))=\nu_p (d)$. Since $e=2\nu_p(d)+1$ the general version of Hensel's Lemma~\cite[3.4.1]{Gouvea} applies, and so there exists $\tilde{x}_1\in \hat{\ZZ}_p$ such that $\tilde{x}_1\equiv x_1 \pmod{p}$ and $f(\tilde{x}_1)=0$, so $c\in \overline{\ell D}$. Thus, if $[0]\in E$, then $E=\ell D_{p^e}$ and $\overline{\ell D}$ is a union of cosets of the form $(c+p^e \hat{\ZZ}_p)$, Proposition~\ref{Prop:union} applies, and $L_{\ell D}=L_{\overline{\ell D}}\in \QQ$.\\ \\

If $E\neq \ell D_{p^e}$ then we claim that $\ell D_{p^e}\backslash E=\{ [0]\}$. If $[c]\in \ell D_{p^e}\backslash E$ then there exists $\{x_i  \}_{i=1}^\ell$ such that $c=\displaystyle \sum_{i=1}^\ell x_i^d\pmod{p^e}$ and $p\mid x_i$ for all $i$. This implies that $p^d\bigg\vert \displaystyle \sum_{i=1}^\ell x_i^d$ and so $p^d \mid c$, hence $[c]=[0]$. Assume that $d\neq 2,4$ and $p\neq 2$, then $d\geq e$. Let $x_i=p\cdot \tilde{x}_i$ and let $\tilde{c}=\displaystyle  \sum_{i=1}^\ell \tilde{x}_i^d$. We then have $c=p^d \tilde{c}$ with $\tilde{c}\in \overline{\ell D}$. Conversely if $\tilde{c}\in \overline{\ell D}$, then $c=p^d\tilde{c} \in  \overline{\ell D}$ and $c\equiv 0 \pmod{p^e}$. Thus $\overline{\ell D}=\left(\bigcup (c+p^e\hat{\ZZ}_p) \right)\cup p^d \overline{\ell D}$, where the union is over cosets for which $[c]\in E$. Corollary~\ref{Cor:unionquadratic} applies here to show that $L_{\ell D,p}=L_{\overline{\ell D}}$ is the root of a quadratic polynomial over $\QQ[x]$. Assume now that $p=2$. Then, $L_{\ell D,2}$ is also a root of a quadratic polynomial by \cite[Theorem 3]{Fares} when $d=2$ and by Proposition~\ref{alphapoly} below when $d=4$.
\end{proof}

\noindent
\begin{Cor}
For a fixed $\ell$, if $d$ is odd and $p$ is a prime, then $L_{\ell D,p} \in \QQ$.
\end{Cor}
\begin{proof}
Write $0=x^d+(-x)^d$, where $p\nmid x$, hence $L_{\ell D,p} \in \QQ$ by Theorem~\ref{Th:main}.
\end{proof}

\noindent
Proposition~\ref{Prop:continued} and Corollary~\ref{Cor:unionquadratic} give us algorithms that can be used to obtain $L_{\ell D,p}$ in either of the cases above.\\ 

\noindent
The rest of this work will describe cases in which we can determine whether or not $0$ can be non-trivially written as the sum of $\ell$ elements to the power of $d$ and so determine the valuative capacity. 

\section{Specific valuative capacities}

\noindent
To begin this section we need to recall a definition and result from~\cite{Small}.

\begin{Def}
Given an integer $d$, a prime $p$ and an integer $e>1$, the Waring number $g(d,p^e) \pmod{p^e}$, is the smallest integer such that every element of $\ZZ/(p^e)$ can be written as a sum of $g(d,p^e)$ elements to the power of $d$.
\end{Def}

\noindent
\begin{Lemma}\label{Lemma:t}
Let $d=2^\alpha \beta$, $\alpha\geq 0$, $\beta$ odd, and let $p$ be an odd prime. Then
\begin{enumerate}
\item $-1$ is a $d$-th power $\pmod{p}$ if and only if $p\equiv 1 \pmod{2^{\alpha +1}}$.
\item If $p\not \equiv 1 \pmod{2^{\alpha +1}}$, then $g(d,p^e)\geq 3$ for all $e>1$.
\end{enumerate}
\end{Lemma}

\subsection{When $p$ is odd}

\noindent
If $p\nmid d$, we have $e=1$, and we obtain the following formula for the valuative capacity:\\

\noindent
\begin{Prop}\label{nicerational}
For a fixed $\ell$, if $p\nmid d$ and $d=2^\alpha \beta$, with $\beta$ odd, if $p\equiv 1 \pmod{2^{\alpha+1}}$ then $$L_{\ell D,p}=\tfrac{1}{|\ell D_p|}\left(1+\tfrac{1}{p-1} \right).$$
\end{Prop}
\begin{proof}
Theorem~\ref{Th:main} gives that $L_{\ell D,p}\in \QQ$, since by Lemma~\ref{Lemma:t}(1), there exist $x\in \ZZ$ such that $x^d\equiv -1 \pmod{p}$, hence $1^d+x^d \equiv 0 \pmod{p}$. Since $p\nmid d$, $e=1$. For any $c\in \ell D_p$, the coset $c+p\ZZ$, has for valuative capacity $L_{(c+p\ZZ),p}=1+\frac{1}{p-1}$ by Proposition~\ref{Prop:cosets}, and then $L_{\ell D,p}=\frac{1}{|\ell D_p|}\left(1+\frac{1}{p-1} \right)$ by Proposition~\ref{Prop:shuffle}.\\ \\
\end{proof}

\noindent
The above means, in particular that when $d$ is odd, we have a rational valuative capacity. When $d$ is even, with $p\nmid d$ and $p\not\equiv 1 \pmod{2^{\alpha+1}}$, we can obtain explicitly the quadratic polynomial for which $L_{\ell D,p}$ is a root.

\begin{Prop}\label{Prop:quadratic}
Let $p$ be odd and $d$ an even integer such that $d=2^{\alpha}\beta$, with $\beta$ odd. If $p\not\equiv 1 \pmod{2^{\alpha+1}}$ and $p\nmid d$, then $L_{\ell D,p}$ is the positive root of the quadratic equation with coefficients in $\QQ$: $$L_{\ell D,p}^2+dL_{\ell D,p}-\tfrac{(p-1)d}{|\ell D_p|}=0.$$
\end{Prop}
\begin{proof}
When $d$ is even, if $p\nmid d$, we have that $e=1$ in the proof of Theorem~\ref{Th:main} and 
$\overline{\ell D}= \left(\bigcup (c+p\hat{\ZZ}_p)\right) \cup p^d \overline{\ell D}$
 for all $c_i$ that can be written as $\ell$ $d$-th powers $\pmod{p}$. 
Thus 
$$L_{\overline{\ell D},p}=L_{\ell D,p}=\cfrac{1}{\cfrac{|\ell D_p|}{p-1}+\cfrac{1}{d+L_{\ell D,p}}}$$ by Propositions~\ref{Prop:shuffle}. Solving for $L_{\ell D,p}$, gives the stated quadratic equation. Its discriminant is $d^2+\frac{4(p-1)d}{|\ell D_p|}$ which is greater than $d^2$, hence the equation only has one positive root. 
\end{proof}

\noindent
Next we look into the case $p>(d-1)^4$, where $p\nmid d$ and a result from~\cite{Small2} gives a very nice formula for the valuative capacity in the case where $d$ is odd, and using the above, we can still get more details about the valuative capacity when $d=2^\alpha \beta$ and $p\equiv 1 \pmod{2^{\alpha+1}}$.

\noindent
\begin{Prop}\label{Th:bigp}
If $p>(d-1)^4$ and $d>2$ then
\begin{enumerate}
\item For $d$ odd, $L_{\ell D,p} =\frac{1}{p-1}$.
\item For $d$ even, with $d=2^\alpha \beta$ and $\beta$ odd: 
\begin{enumerate}
\item If $p\equiv 1 \pmod{2^{\alpha+1}}$, then $g(d,p^e)=2$ and $L_{\ell D,p} =\frac{1}{p-1}$.
\item If $p\not\equiv 1 \pmod{2^{\alpha+1}}$, then $g(d,p^e)=2$ for $e=1$ and $g(d,p^e)= 3$ otherwise, and, since $p \nmid d$,
\begin{enumerate}
\item if $\ell=2$, then $L_{2D,p}$ is the root of the quadratic equation $L_{2D,p}^2+dL_{2D,p}-\tfrac{(p-1)d}{|2D_p|}=0$,
\item if $\ell\geq 3$, then $L_{\ell D,p}=\frac{1}{|\ell D_p|}\left(1+\frac{1}{p-1} \right)$.
\end{enumerate}
\end{enumerate}
\end{enumerate}
\end{Prop}
\begin{proof}
For both $d$ odd and even we have that $g(d,p)\leq 2$, for $p>(d-1)^4$ by \cite{Small}, so this gives us that, $\ell D_p=\ZZ/(p)$. Now a lifting lemma in a paper by the same author ~\cite[2.1]{Small2}, gives us that if $d$ is odd, and if $c\equiv x^d+y^d \pmod{p}$ has a solution, then $c\equiv x^d+y^d \pmod{p^e}$, for any $e>1$, hence $g(d,p^e)=g(d,p)$ and, $D_{p^e}=\ZZ/(p^e)$. Using Proposition~\ref{Prop:integers} we get that $L_{\ell D,p} =\frac{1}{p-1}$.\\ \\
For $d$ even and $p\equiv 1 \pmod{2^{\alpha+1}}$, then Lemma~\ref{Lemma:t}(1), gives that $g(d,p)=2$. Thus $0$ can be written as a sum of non trivial $d$-th powers, and then we can lift any solution of $x^d+y^d\equiv c$ in $\ZZ/(p)$, to a solution in $\ZZ/(p^e)$ for $e>1 \in \ZZ$. Now by \cite{Small}, for $\ell>2$, we have $\ell D=\ZZ/(p)$, hence $L_{\ell D_p,p} =\frac{1}{p-1}$ in this case as well.\\ \\
In the case $d$ even and $p\not\equiv 1 \pmod{2^{\alpha+1}}$, Lemma~\ref{Lemma:t}(2), gives that $g(d,p^e)=2$ for $e=1$ and since $p>(d-1)^4$ by ~\cite{Small} $g(d,p^e)= 3$ otherwise, thus (b)i. is Proposition~\ref{Prop:quadratic} and (b)ii. is Proposition~\ref{nicerational}.
\end{proof}

\noindent
There is one last case of $d$ odd, when $2<p<(d-1)^4$, where we can obtain a nice formula for the valuative capacity:

\noindent
\begin{Th}\label{Th:less}
For $p$ a prime, with $p\nmid d$, and $d>2$,  if $\gcd(d,p-1)=1$, then $D= \ZZ/(p^e)$ for $e\geq 1$, and for $\ell>1$, $L_{\ell D,p} =\frac{1}{p-1}$.
\end{Th}
\begin{proof}
For both $d$ odd and even we have that the $d$-th power map $\ZZ/(p) \to \ZZ/(p)$,
is onto if $gcd(d,p-1)=1$. Thus, in this case $D=\ZZ/(p)$, and so $\ell D=\ZZ/(p)$ also. Note that in this case $0$ can be non-trivially written as the sum of two elements to the power of $d$. This allows us to use a lifting Lemma~\cite[2.1]{Small2}, giving us that if $d$ is odd, and $c\equiv x^d+y^d \pmod{p}$ has a solution, then $c\equiv x^d+y^d \pmod{p^e}$, for $e>1$, hence $g(d,p^e)=g(d,p)$ and $\ell D=\ZZ/(p^e)$. Thus $L_{\ell D,p} =\frac{1}{p-1}$ as in Proposition~\ref{Th:bigp}. \\ \\
\end{proof}

\noindent
Note that for the previous theorem the case $\ell=1$ can be found in~\cite{Fares1}.

\subsection{When $p=2$}

\noindent
We also look into the case $p=2$, where we might not find explicit formulas for valuative capacities, but we look at how one would search for them and we get different proofs of known results. 
In order to compute the valuative capacity when $p=2$ and $d$ even, we need to establish something similar to Hensel's Lemma to allow us to lift values to $\ZZ/(2^e)$, with $e>1$.

\noindent
\begin{Prop}\label{Prop:2powers}
Let $a$ be an odd integer and $e>1$, then 
\begin{enumerate}
\item $x^{2^e}\equiv a \pmod{2}$ has a solution for all $a$.
\item For $n\leq e+2$, $x^{2^e}\equiv a \pmod{2^n}$ has a solution if and only if $a\equiv 1 \pmod{2^n}$.
\item For $n> e+2$, $x^{2^e}\equiv a \pmod{2^n}$ has a solution if and only if $a\equiv 1 \pmod{2^{e+2}}$.
\end{enumerate}
\end{Prop}

\noindent
\begin{Prop}\label{Prop:odd1} For $d$ an odd integer, $n\in \NN$, the image of the $d$-th power map on odd values is $ \{2m+1 \pmod{2^n} \ | \  m \in \ZZ \}$.
\end{Prop}

\noindent
Using the above we can now figure out which elements can be lifted, since the odd powers have the same characterization as the $2^\alpha$th powers.

\noindent
\begin{Prop}\label{kpower2}
If $d=2^\alpha \beta$, where $\alpha\geq 1$ and $\beta$ is an odd integer $\geq 1$, then we can write $\bard$ in the following way:
$$\bard=\{ 0 \} \cup (1+2^{\alpha+2}\hat{\ZZ}_2)\cup 2^d (1+2^{\alpha+2}\hat{\ZZ}_2)\cup 2^{2d} (1+2^{\alpha+2}\hat{\ZZ}_2) \cup 2^{3d} (1+2^{\alpha+2}\hat{\ZZ}_2)\ldots .$$
\end{Prop}
\begin{proof}
Using Proposition~\ref{Prop:2powers}, we get that the odd $d$-th powers have the same characterization as the $2^\alpha$-th, since they are the $2^\alpha$-th powers of odd values, which are all the odd values since Proposition~\ref{Prop:2powers}, can be used to show that this maps onto odd values. The even $d$-th powers are of the form $2^{dn}c$, where $n\in \NN$ and $c$ is an odd $d$-th power.  Hence the result.
\end{proof}

\noindent
\begin{Prop}\label{alphapoly}
For $\ell=2$ and any $d=2^\alpha \beta$, where $\alpha\geq 1$ and $\beta$ is an odd integer $\geq 1$, $L_{D+D}$ is the positive root of the following polynomial depending on $d$: 
$$(2\alpha+6)L^2+(2\alpha d-2\alpha +6d -7)L+(\alpha+3-\alpha d^2 -6\alpha d -9d)=0 .$$
\end{Prop}
\begin{proof}
Using Proposition~\ref{kpower2}, we obtain that
\begin{align*}
\bard+\bard &=\{ 0 \} \cup (\{1,2\}+2^{\alpha+2}\hat{\ZZ}_2)\cup 2^d (\{1,2\}+2^{\alpha+2}\hat{\ZZ}_2)\cup 2^{2d} (\{1,2\}+2^{\alpha+2}\hat{\ZZ}_2) \cup 2^{3d} (\{1,2\}+2^{\alpha+2}\hat{\ZZ}_2)\ldots\\
&=\bard\cup 2\bard.
\end{align*}
Using Proposition~\ref{Prop:shuffle} we obtain $\dfrac{1}{L_{\bard+\bard}}=\dfrac{1}{L_{\bard}}+\dfrac{1}{L_{2\bard\cup 2^d(\bard+\bard)}}$. Using Proposition~\ref{Prop:cosets} and~\ref{Prop:Johnson}, we get the above polynomial.
\end{proof}

\noindent
For the next we will visit the case $d=2$, note that our results to coincides with the ones from~\cite{Fares}. Our results can also be used to generalize the following theorem of Legendre on sums of squares:

\noindent
\begin{Prop}\label{Prop:Legendre}
When $\ell=3$ and $d=2$, $$\overline{3D}_2=\{0 \} \cup\bigcup_{i=0}^\infty 2^{2i}(\{1,2,3,4,5,6 \} + 8\hat{\ZZ}_2).$$
\end{Prop}
\begin{proof}
We have shown in Proposition~\ref{kpower2} that $\overline{D}_2 =\{0 \} \cup \bigcup_{i=0}^\infty 2^i(1+8\hat{\ZZ}_2)$. When adding the cosets triple-wise, we get $$\overline{3D}_2=\{0 \} \cup\bigcup_{i=0}^\infty 2^{2i}(\{1,2,3,4,5,6 \} + 8\hat{\ZZ}_2) .$$
The only elements not in $\overline{3D}_2$ are those of the form $2^{2i} (7+8\hat{\ZZ}_2)$, which corresponds to Legendre's theorem.
\end{proof}

\noindent
\begin{Prop}
If $d=2$, $\ell\geq 4$ and $n\geq 1$, we have that $\overline{\ell D}_{2^n} =\ZZ/(2^n)$ and $L_{\overline{\ell D}_{2}}=1$.
\end{Prop}
\begin{proof}
By Proposition~\ref{Prop:Legendre}, the only cosets missing are those of the form $2^i (7+8\hat{\ZZ}_2)$, which can now be obtained since 7 can be written as the sum of 4 squares, $7=4+1+1+1$. Thus $\overline{\ell D}_{2^n} =\ZZ/(2^n)$. By Proposition~\ref{Prop:cosets} $L_{\overline{3D_2}}=L_{3D}=\frac{1}{p-1}=1$.
\end{proof}\ \\

\noindent
To conclude this paper we have added a  table of various other valuative capacities ($L$) for $3D$, for both odd and even $p$:
\begin{center}
\begin{tabular}{| c | c | c | c | }
\hline
$p$ & $d$ & $e=2\nu_p(d)+1$ & $L$\\
\hline
$2$ & $2$ & $3$ & $\frac{21}{22}$\\[0.2cm]
$2$ & $4$ & $5$ & $\frac{3}{2}$\\[0.2cm]
$2$ & $6$ & $3$ & $\frac{5}{4}$\\[0.2cm]
$2$ & $8$ & $7$ & $\frac{14}{15}$\\[0.2cm]
$3$ & $6$ & $3$ & $\frac{155}{204}$\\[0.2cm]
$3$ & $12$ & $3$ & $\frac{155}{204}$\\[0.2cm]
$3$ & $18$ & $5$ & $\frac{511}{488}$\\[0.2cm]
$3$ & $27$ & $7$ & $\frac{143}{170}$\\
\hline
\end{tabular}
\end{center}


\newpage
\begin{bibdiv}
\begin{biblist}


\bib{Bh}{article}{
   author={Bhargava, M.},
   title={$p$-orderings and polynomial functions on arbitrary subsets of Dedekind rings},
   journal={Journal: Fur Die Reine Und Angewandte Mathematik}, 
   volume={490},
   pages={101-127},
   date={1997},
}

\bib{CB}{article}{
   author={Boulanger, J.},
   author={Chabert, J.-L.},
   title={Asymptotic Behavior of Characteristic Sequences of Integer-Valued Polynomials},
   journal={Journal of Number Theory}, 
   volume={80},
   pages={238-259},
   date={2000}
}

\bib{CC}{book}{
   author={Cahen, P.-J.},
   author={Chabert, J.-L.},
   title={Integer-valued polynomials},
   date={1997},
   publisher={American Mathematical Society}, 
   volume={48}
   address={Providence, Rhode Island}    
}

\bib{Chabert}{article}{
   author={Chabert, J.-L.},
   title={Generalized factorial ideals},
   journal={The Arabian Journal for Science and Engineering}, 
   volume={26},
   pages={51-68},
   date={2001}
}

\bib{Davenport}{book}{
   author={Davenport, H.},
   title={The Higher Arithmetic},
   date={1982},
   publisher={Cambridge University Press}, 
}

\bib{Fares1}{article}{
   author={Fares, Y.},
   author={Johnson, K.},
   title={The characteristic sequence and $p$-orderings of the set of $d$-th powers of integers},
   journal={Integers, 12, no. 5, },
   date={2012},
}


\bib{Fares}{article}{
   author={Fares, Y.},
   author={Johnson, K.},
   title={The Valuative Capacities of the Sets of Sums of Two and of Three Squares},
   journal={Integers, volume 16},
   date={2016},
}

\bib{Fares2}{article}{
   author={Fares, Y.},
   author={Petite, S.},
   title={The Valuative Capacity of Subshifts of Finite Type},
   journal={Journal of Number Theory}, 
   volume={158},
    pages={165-184},
   date={2016},
}

\bib{Fekete}{article}{
   author={Fekete, M.},
   title={\"{U}ber die Verteilung der Wurzeln bei gewissen algebraischen Gleichungen mit ganzzahligen Koeffizienten},
   journal={Math. Z. no. 1},
   pages={228–249},
   date={1923},
}


\bib{Ferguson}{article}{
   author={Ferguson, L. B. O.},
   title={What can be approximated by polynomials with integer coefficients},
   journal={Amer. Math. Monthly, no.5, 113},
   pages={403–414},
   date={2006},
}


\bib{Gouvea}{book}{
   author={Gouv\^{e}a, F.},
   title={$p$-adic numbers: An introduction},
   date={1997},
   publisher={second ed., Universitext, Springer-Verlag}, 
   address={Berlin}    
}

\bib{Johnson2}{article}{
   author={Johnson, K.},
   title={Limits of characteristic sequences of integer-valued polynomials on homogeneous sets},
   journal={Journal of Number Theory 129},
   pages={2933–2942},
   date={2009},
}

\bib{Johnson1}{article}{
   author={Johnson, K.},
   title={$p$-Orderings of finite subsets of Dedekind domains},
   journal={J. Algebraic Combin 30, no.2},
   pages={233-253},
   date={2009},
   
}

\bib{Johnson}{article}{
   author={Johnson, K.},
   title={$p$-Orderings of Noncommutative Rings},
   journal={Proc. AMS 143, no. 8,},
   pages={3265–3279},
   date={2015},
}


\bib{PS}{book}{
   author={Polya, G.},
   author={Szeg\"{o}, G.},
   title={Problems and theorems in analysis. Vol. I: Series, integral calculus, theory of functions.},
   date={1972},
   publisher={Translated from the German by D. Aeppli Die Grundlehren der mathematischen Wissenschaften, Band 193. Springer-Verlag},
}

\bib{Rumely}{book}{
   author={Rumely, R.},
   title={Capacity theory with local rationality. The strong Fekete-Szeg\"{o} theorem on curves},
   date={2013},
   publisher={ Mathematical Surveys and Monographs, American Mathematical Society}, 
   volume={193}
   address={Providence, Rhode Island}    
}

\bib{Small}{article}{
   author={Small, C.},
   title={Solution of Waring's problem $\pmod{n}$},
   journal={ Amer. Math. Monthly 84, no. 5},
   pages={356–359},
   date={1977},
}

\bib{Small2}{article}{
   author={Small, C.},
   title={Waring's problem $\pmod{n}$},
   journal={ Amer. Math. Monthly 84, no. 1},
   pages={12–25},
   date={1977},
}

\end{biblist}
\end{bibdiv}
\end{document}